\documentclass[11pt]{amsart}
\usepackage{amsmath,amssymb,amsfonts,amsbsy}
\usepackage{graphicx}
\usepackage{float}
\usepackage{setspace}
\usepackage[margin=1.0in]{geometry}

\newtheorem{thm}{Theorem}[section]

\newtheorem{cor}[thm]{Corollary}
\newtheorem{prop}[thm]{Proposition}

\theoremstyle{definition}
\newtheorem{defn}[thm]{Definition}
\theoremstyle{definition}
\newtheorem{remark}[thm]{Remark}

\newcommand{\mc}[1]{\mathcal{#1}}
\newcommand{\e}[1]{\emph{#1}}
\newcommand{\la}{\langle}
\newcommand{\ra}{\rangle}

\newcommand{\rmv}[1]{}

\newcommand{\hs}{\hskip10pt}

\newcommand{\sq}{\square}
\newcommand{\LUC}{\mathrm{LUC}(\mathbb{G})}
\newcommand{\RUC}{\mathrm{RUC}(\mathbb{G})}

\newcommand{\Htm}{\mc{H}_{\Theta(m)}}
\newcommand{\LG}{\mc{L}(G)}
\newcommand{\RG}{\mc{R}(G)}
\newcommand{\LO}{L_1(G)}
\newcommand{\LOQ}{L_1(\mathbb{G})}
\newcommand{\LOQH}{L_1(\hat{\mathbb{G}})}

\newcommand{\LOQHP}{L_1(\hat{\mathbb{G}}')}

\newcommand{\LTQ}{L_2(\mathbb{G})}
\newcommand{\LI}{L_{\infty}(G)}
\newcommand{\LIQ}{L_{\infty}(\mathbb{G})}
\newcommand{\LIQH}{L_{\infty}(\hat{\mathbb{G}})}
\newcommand{\LIQHP}{L_{\infty}(\hat{\mathbb{G}}')}
\newcommand{\BH}{\mc{B}(H)}

\newcommand{\BLT}{\mc{B}(L_2(G))}
\newcommand{\BLTQ}{\mc{B}(L_2(\mathbb{G}))}
\newcommand{\TC}{\mc{T}(L_2(G))}
\newcommand{\TCQ}{\mc{T}(L_2(\mathbb{G}))}

\newcommand{\CBLTQ}{\mc{CB}_{L_{\infty}(\hat{\mathbb{G}})}^{L_{\infty}(\mathbb{G})}(\mc{B}(L_2(\mathbb{G})))}
\newcommand{\CBTCrr}{\mc{CB}_{\mc{T}_{\rhd}}(\mc{B}(L_2(\mathbb{G})))}
\newcommand{\CBTClr}{_{\mc{T}_{\rhd}}\mc{CB}(\mc{B}(L_2(\mathbb{G})))}
\newcommand{\CBTCll}{_{\mc{T}_{\lhd}}\mc{CB}(\mc{B}(L_2(\mathbb{G})))}
\newcommand{\CBTCrl}{\mc{CB}_{\mc{T}_{\lhd}}(\mc{B}(L_2(\mathbb{G})))}

\newcommand{\Mcb}{M_{cb}A(G)}

\newcommand{\vphi}{\varphi}
\newcommand{\lm}{\lambda}
\newcommand{\Gam}{\Gamma}
\newcommand{\om}{\omega}

\newcommand{\ten}{\otimes}
\newcommand{\oten}{\bar{\otimes}}
\newcommand{\hten}{\widehat{\otimes}}
\newcommand{\id}{\iota}
\newcommand{\h}[1]{\hat{#1}}

\providecommand{\norm}[1]{\lVert#1\rVert}

\newcommand{\G}{\mathbb{G}}
\newcommand{\C}{\mathbb{C}}

\newcommand{\R}{\mathbb{R}}

\begin{document}

\title{Amenability and covariant injectivity of locally compact quantum groups}
\author{Jason Crann$^{1,2}$ and Matthias Neufang$^{1,2}$}
\email{jason\_crann@carleton.ca, Matthias.Neufang@math.univ-lille1.fr}
\address{$^1$School of Mathematics \& Statistics, Carleton University, Ottawa, ON, Canada K1S 5B6}
\address{$^2$Universit\'{e} Lille 1 - Sciences et Technologies, UFR de Math\'{e}matiques, Laboratoire de Math\'{e}matiques Paul Painlev\'{e} - UMR CNRS 8524, 59655 Villeneuve d'Ascq C\'{e}dex, France}

\begin{spacing}{1.5}

\begin{abstract} As is well known, the equivalence between amenability of a locally compact group $G$ and injectivity of its von Neumann algebra $\LG$ does not hold in general beyond inner amenable groups. In this paper, we show that the equivalence persists for all locally compact groups if $\LG$ is considered as a $\TC$-module with respect to a natural action. In fact, we prove an appropriate version of this result for every locally compact quantum group.\end{abstract}

\maketitle

\section{Introduction}


The connection\let\thefootnote\relax\footnotetext{2010 \e{Mathematics Subject Classification:} Primary 22D15 46L89 81R15, Secondary 43A07 46M10 43A20.} between amenability of a locally compact group $G$ and injectivity of the von Neumann algebra $\LG$ associated with the left regular representation has been a topic of interest in abstract harmonic analysis for decades. Amenability of $G$ entails injectivity of $\LG$, however, the converse is not true, e.g., if $G=SL_n(\R)$ for $n\geq2$; indeed, a result of Connes' \cite[Corollary 7]{C}, attributed to Dixmier, states that $\LG$ is injective for any separable connected locally compact group. In order to find a strengthening of injectivity which would be equivalent to amenability, there have been two main approaches: in terms of additional properties of the underlying group, or of the associated conditional expectations. In the spirit of the first approach, Lau and Paterson showed that $G$ is amenable if and only if $\LG$ is injective and $G$ is inner amenable \cite[Corollary 3.2]{LP}. Following the second approach, So\l tan and Viselter recently proved, in the more general setting of locally compact quantum groups $\G$, that amenability is equivalent to the existence of a conditional expectation $E:\BLTQ\rightarrow\LIQH$ which maps $\LIQ$ into the center of $\LIQH$ \cite[Theorem 3]{SV}. In the present paper, we provide a new perspective on this connection, even at the level of locally compact quantum groups, by presenting new characterizations of amenability using the $\TCQ$-module structure of $\BLTQ$. We note that the use of an action by $\TCQ$, rather than one by $\LOQ$, is crucial.

We begin in section 2 by recalling the relevant definitions and results from the theory of locally compact quantum groups, as introduced by Kustermans and Vaes \cite{KV1,KV2,V}.

Section 3 is devoted to an overview of the $\TCQ$-bimodule structures on $\BLTQ$ and its relation to the spaces $\LUC$ and $\RUC$ of left and right uniformly continuous functionals on a locally compact quantum group $\G$, as introduced in \cite{HNR2,R}. For any locally compact quantum group $\G$, there are two canonical completely contractive Banach algebra structures on $\TCQ$, denoted by $(\TCQ,\lhd)$ and $(\TCQ,\rhd)$, induced by the left and right fundamental unitaries of $\G$, respectively. This in turn yields two interesting bimodule structures on $\BLTQ$, which have been a recent topic of interest in the development of harmonic analysis on locally compact quantum groups \cite{HNR1,HNR3}, and are closely related to $\LUC$ and $\RUC$.

The dual space of $\LUC$ carries a natural Banach algebra structure. In \cite{HNR1}, Hu, Neufang and Ruan studied various properties of this algebra, in particular through a weak*-weak* continuous, injective, completely contractive representation
\begin{equation*}\Theta^r:\LUC^*\rightarrow\CBTCrr\end{equation*}
in the algebra of completely bounded right $(\TCQ,\rhd)$-module maps on $\BLTQ$. This representation is the fundamental tool in our work, and is used in section 4 to show that a locally compact quantum group $\G$ is amenable if and only if the dual quantum group $\h{\G}$ is what we shall call \e{covariantly injective}, meaning the corresponding projection of norm one commutes with the module action of $(\TCQ,\rhd)$ on $\BLTQ$. As an application, we obtain a new proof of the recently answered question of B\'{e}dos and Tuset concerning the topological amenability of $\G$ \cite{Z}. By examining the remaining three $\TCQ$-module structures on $\BLTQ$, we obtain new characterizations of amenability, co-commutativity, as well as injectivity of $\h{\G}$. Moreover, compactness of $\G$ can be characterized in terms of normal conditional expectations respecting the $\TCQ$-module structure.

We finish in section 5 by showing that a locally compact quantum group $\G$ is amenable if and only if $\LIQH$ is an injective operator $\TCQ$-module. Even in the commutative case, this provides a new identification of classical amenability of a locally compact group $G$ in terms of the injectivity of $\LG$ as a $\TC$-module. We also show that both amenability of $\G$ \e{and} of $\h{\G}$ may be characterized through the injectivity of $\BLTQ$ as a left, respectively, right $\TCQ$-module. This, along with other results in the paper suggests that these homological methods may provide a new approach to the duality problem of amenability and co-amenability for arbitrary locally compact quantum groups.

\section{Preliminaries}

A \e{locally compact quantum group} is a quadruple $\G=(\LIQ,\Gam,\vphi,\psi)$, where $\LIQ$ is a Hopf-von Neumann algebra with a co-associative co-multiplication $\Gam:\LIQ\rightarrow\LIQ\oten\LIQ$, and $\vphi$ and $\psi$ are fixed (normal faithful semifinite) left and right Haar weights on $\LIQ$, respectively \cite{KV2,V}. For every locally compact quantum group $\G$, there exists a \e{left fundamental unitary operator} $W$ on $L_2(\G,\vphi)\ten L_2(\G,\vphi)$ and a \e{right fundamental unitary operator} $V$ on $L_2(\G,\psi)\ten L_2(\G,\psi)$ implementing the co-multiplication $\Gam$ via
\begin{equation*}\Gam(x)=W^*(1\ten x)W=V(x\ten 1)V^*\hs(x\in\LIQ).\end{equation*}
Both unitaries satisfy the \e{pentagonal relation}; that is,
\begin{equation}\label{penta}W_{12}W_{13}W_{23}=W_{23}W_{12}\hs\mathrm{and}\hs V_{12}V_{13}V_{23}=V_{23}V_{12}.\end{equation}
By \cite[Proposition 2.11]{KV2}, we may identify $L_2(\G,\vphi)$ and $L_2(\G,\psi)$, so we will simply use $\LTQ$ for this Hilbert space throughout the paper.

Let $\LOQ$ denote the predual of $\LIQ$. Then the pre-adjoint of $\Gam$ induces an associative completely contractive multiplication on $\LOQ$, defined by
\begin{equation*}\star:\LOQ\hten\LOQ\ni f\ten g\mapsto f\star g=\Gam_*(f\ten g)\in\LOQ.\end{equation*}
The multiplication $\star$ is a complete quotient map from $\LOQ\hten\LOQ$ onto $\LOQ$, implying
\begin{equation*}\la\LOQ\star\LOQ\ra=\LOQ,\end{equation*}
where $\la\LOQ\star\LOQ\ra$ denotes the closed linear span of $f\star g$, with $f,g\in\LOQ$. There is a canonical $\LOQ$-bimodule structure on $\LIQ$, defined by
\begin{equation*}\la f\star x,g\ra=\la x,g\star f\ra\hs\mathrm{and}\hs\la x\star f,g\ra=\la x,f\star g\ra,\end{equation*}
for $x\in\LIQ$, and $f,g\in\LOQ$. Using the co-multiplication $\Gam$ we may write
\begin{equation*}f\star x=(\id\ten f)\Gam(x)\hs\mathrm{and}\hs x\star f=(f\ten\id)\Gam(x)\hs(x\in\LIQ, f\in\LOQ).\end{equation*}
If $X$ is an operator system in $\LIQ$ that is also a left $\LOQ$-submodule, then a \e{left invariant mean on $X$}, is a state $m\in X^*$ satisfying
\begin{equation}\label{leftinv}\la m,f\star x\ra=\la f,1\ra\la m,x\ra\hs(x\in X, f\in\LOQ).\end{equation}
Right and two-sided invariant means are defined similarly. A locally compact quantum group $\G$ is said to be \e{amenable} if there exists a left invariant mean on $\LIQ$. It is known that $\G$ is amenable if and only if there exists a right (equivalently, two-sided) invariant mean (cf. \cite[Proposition 3]{DQV}). $\G$ is said to be \e{co-amenable} if $\LOQ$ has a bounded left (equivalently, right or two-sided) approximate identity (cf. \cite[Theorem 3.1]{BT}).

Given a locally compact quantum group $\G$, the \e{left regular representation} $\lm:\LOQ\rightarrow\BLTQ$ is defined by
\begin{equation*}\lm(f)=(f\ten\id)(W)\hs(f\in\LOQ),\end{equation*}
and is an injective, completely contractive homomorphism from $\LOQ$ into $\BLTQ$. Then $\LIQH:=\{\lm(f) : f\in\LOQ\}''$ is the von Neumann algebra associated with the dual quantum group $\h{\G}$ of $\G$. Analogously, we have the \e{right regular representation} $\rho:\LOQ\rightarrow\BLTQ$ defined by
\begin{equation*}\rho(f)=(\id\ten f)(V)\hs(f\in\LOQ),\end{equation*}
which is also an injective, completely contractive homomorphism from $\LOQ$ into $\BLTQ$. Then $\LIQHP:=\{\rho(f) : f\in\LOQ\}''$ is the von Neumann algebra associated to the quantum group $\h{\G}'$. It follows that $\LIQHP=\LIQH'$, and the fundamental unitaries satisfy $W\in\LIQ\oten\LIQH$ and $V\in\LIQHP\oten\LIQ$ \cite[Proposition 2.15]{KV2}. Moreover, dual quantum groups always satisfy $\LIQ\cap\LIQH=\LIQ\cap\LIQHP=\C1$ \cite{VD}.

If $G$ is a locally compact group, we let $\G_a=(\LI,\Gam_a,\vphi_a,\psi_a)$ denote the \e{commutative} quantum group associated with the commutative von Neumann algebra $\LI$, where the co-multiplication is given by $\Gam_a(f)(s,t)=f(st)$, and $\vphi_a$ and $\psi_a$ are integration with respect to a left and right Haar measure, respectively. The dual quantum group $\h{\G}_a$ of $\G_a$ is the \e{co-commutative} quantum group $\G_s=(\LG,\Gam_s,\vphi_s,\psi_s)$, where $\LG$ is the left group von Neumann algebra with co-multiplication $\Gam_s(\lm(t))=\lm(t)\ten\lm(t)$, and $\vphi_s=\psi_s$ is Haagerup's Plancherel weight (cf. \cite[\S VII.3]{T2}). Here, $\G_s$ is called co-commutative since its co-multiplication is symmetric. We also consider the quantum group $\h{\G}_a'=\G_s'$ associated to the right group von Neumann algebra $\RG$ with the co-multiplication $\Gam_s'(\rho(t))=\rho(t)\ten\rho(t)$. Then $L_1(\G_a)$ is the usual group convolution algebra $\LO$, and $L_1(\G_s)=L_1(\G_s')$ is the Fourier algebra $A(G)$. It is known that every commutative locally compact quantum group is of the form $\G_a$ \cite[Theorem 2; \S2]{T,VV}. Therefore, every commutative locally compact quantum group is co-amenable, and is amenable if and only if the underlying locally compact group is amenable. By duality, every co-commutative locally compact quantum group is of the form $\G_s$, which is always amenable \cite[Theorem 4]{Re}, and is co-amenable if and only if the underlying locally compact group is amenable, by Leptin's classical theorem.

By using the regular representations of the quantum groups $\h{\G}$ and $\h{\G}'$, we arrive at the \e{reduced quantum group $C^*$-algebra} of $\LIQ$, defined as
\begin{equation*}C_0(\G)=\overline{\h{\lm}(\LOQH)}^{\norm{\cdot}}=\overline{\h{\rho}(\LOQHP)}^{\norm{\cdot}}.\end{equation*}
$\G$ is said to be \e{compact} if $C_0(\G)$ is a unital $C^*$-algebra. For quantum groups arising from locally compact groups $G$, it follows that $C_0(\G_a)$ is $C_0(G)$, the algebra of continuous functions vanishing at infinity, and $C_0(\G_s)$ is the left group $C^*$-algebra $C^*_{\lm}(G)$. The multiplier algebra of $C_0(\G)$ will be denoted $M(C_0(\G))$.

\section{$\LUC$ and $\LUC^*$}

Let $\G$ be a locally compact quantum group. The right fundamental unitary $V$ of $\G$ induces a co-associative co-multiplication
\begin{equation*}\Gam^r:\BLTQ\ni x\mapsto V(x\ten 1)V^*\in\BLTQ\oten\BLTQ,\end{equation*}
and the restriction of $\Gam^r$ to $\LIQ$ yields the original co-multiplication $\Gam$ on $\LIQ$. The pre-adjoint of $\Gam^r$ induces an associative completely contractive multiplication on $\TCQ$, defined by
\begin{equation*}\rhd:\TCQ\hten\TCQ\ni\om\ten\tau\mapsto\om\rhd\tau=\Gam^r_*(\om\ten\tau)\in\TCQ,\end{equation*}
where $\hten$ denotes the operator space projective tensor product. Analogously, the left fundamental unitary $W$ of $\G$ induces a co-associative co-multiplication
\begin{equation*}\Gam^l:\BLTQ\ni x\mapsto W^*(1\ten x)W\in\BLTQ\oten\BLTQ,\end{equation*}
and the restriction of $\Gam^l$ to $\LIQ$ is also equal to $\Gam$. The pre-adjoint of $\Gam^l$ induces another associative completely contractive multiplication \begin{equation*}\lhd:\TCQ\hten\TCQ\ni\om\ten\tau\mapsto\om\lhd\tau=\Gam^l_*(\om\ten\tau)\in\TCQ.\end{equation*}
These two products on $\TCQ$ are quite different in general. It is known that $(\TCQ,\rhd)$ is always left faithful, and right faithful if and only if $\G$ is trivial. Similarly, $(\TCQ,\lhd)$ is always right faithful, and is left faithful if and only if $\G$ is trivial (cf. \cite{HNR1}).

For commutative and co-commutative quantum groups, this type of multiplicative structure on $\TCQ$ has been studied in \cite{A,N,NRS,NR,P}, and the general case has been investigated in \cite{HNR1,HNR3,KN}. In particular, it was shown in \cite[Lemma 5.2]{HNR1} that the pre-annihilator $\LIQ_{\perp}$ of $\LIQ$ in $\TCQ$ is a norm closed two sided ideal in $(\TCQ,\rhd)$ and $(\TCQ,\lhd)$, respectively, and the complete quotient map
\begin{equation}\label{pi}\pi:\TCQ\ni\om\mapsto f=\om|_{\LIQ}\in\LOQ\end{equation}
is a completely contractive algebra homomorphism from $(\TCQ,\rhd)$ and $(\TCQ,\lhd)$, respectively, onto $\LOQ$. Therefore, we have the completely isometric Banach algebra identifications
\begin{equation*}(\LOQ,\star)\cong(\TCQ,\rhd)/\LIQ_{\perp}\hs\mathrm{and}\hs(\LOQ,\star)\cong(\TCQ,\lhd)/\LIQ_{\perp}.\end{equation*}
This allows us to view each of $(\TCQ,\rhd)$ and $(\TCQ,\lhd)$ as a lifting of $(\LOQ,\star)$.

The multiplication $\rhd$ defines a completely contractive $(\TCQ,\rhd)$-bimodule structure on $\BLTQ$ via
\begin{align*}&\BLTQ\hten\TCQ\ni(x,\om)\mapsto x\rhd\om=(\om\ten\id)V(x\ten 1)V^*\in\LIQ\subseteq\BLTQ;\\
              &\TCQ\hten\BLTQ\ni(\om,x)\mapsto \om\rhd x=(\id\ten\om)V(x\ten1)V^*\in\BLTQ.\end{align*}
Note that since $V\in\LIQHP\oten\LIQ$, the bimodule action on $\LIQH$ becomes rather trivial. Indeed, for $\h{x}\in\LIQH$ and $\om\in\TCQ$
\begin{equation}\label{trac}\h{x}\rhd\om=(\om\ten\id)V(\h{x}\ten 1)V^*=\la\om,\h{x}\ra1\hs\mathrm{and}\hs\om\rhd\h{x}=(\id\ten\om)V(\h{x}\ten 1)V^*=\la\om,1\ra\h{x}.\end{equation}

\begin{remark} Observe that the left action of $\TCQ$ on $\LIQ$ satisfies $\om\rhd x=\pi(\om)\star x$, i.e., it is implemented by a left $\LOQ$ action. However, the homological properties of the resulting right action on $\LOQ$ are not equivalent to those corresponding to the canonical right action of $\LOQ$ on itself. For instance, $\LO$ is always right projective over itself for any locally compact group $G$, while it is projective as a right $\TC$-module if and only if $G$ is discrete. See \cite[Theorem 3.3.32]{D} and \cite[Theorem 3.4]{P} for details.\end{remark}

The multiplication $\lhd$ defines, analogously, a completely contractive $(\TCQ,\lhd)$-bimodule structure on $\BLTQ$ via
\begin{align*}&\TCQ\hten\BLTQ\ni(\om,x)\mapsto \om\lhd x=(\id\ten\om)W^*(1\ten x)W\in\LIQ\subseteq\BLTQ;\\
              &\BLTQ\hten\TCQ\ni(x,\om)\mapsto x\lhd\om=(\om\ten\id)W^*(1\ten x)W\in\BLTQ.\end{align*}
In particular, for $x\in\LIQ$ and $f=\om|_{\LIQ}$ with $\om\in\TCQ$, we have
\begin{equation}\label{mod1}x\rhd\om=x\lhd\om=(\om\ten\id)\Gam(x)=x\star f\hs\mathrm{and}\hs\om\lhd x=\om\rhd x=(\id\ten\om)\Gam(x)=f\star x.\end{equation}
As above, we see that the bimodule actions of $(\TCQ,\rhd)$ and $(\TCQ,\lhd)$ on $\BLTQ$ are liftings of the usual bimodule action of $\LOQ$ on $\LIQ$.

If $\G$ is a locally compact quantum group, the subspaces $\LUC$ and $\RUC$ of $\LIQ$ are defined by \cite{HNR2,R}
\begin{equation*}\LUC=\la\LIQ\star\LOQ\ra\hs\mathrm{and}\hs\RUC=\la\LOQ\star\LIQ\ra.\end{equation*}
It was shown by Runde in \cite[Theorem 2.4]{R} that $\LUC$ and $\RUC$ are operator systems in $\LIQ$ such that
\begin{equation}\label{MC0}\C_0(\G)\subseteq\LUC,\RUC\subseteq M(C_0(\G)).\end{equation}
In the classical setting of locally compact groups $G$, $\mathrm{LUC}(\G_a)$ (respectively, $\mathrm{RUC}(\G_a)$) is the usual space $\mathrm{LUC}(G)$ (respectively, $\mathrm{RUC}(G)$) of bounded left (respectively, right) uniformly continuous functions on $G$, and $\mathrm{LUC}(\G_s)=\mathrm{RUC}(\G_s)$ is the space $\mathrm{UCB}(\h{G})$ of uniformly continuous linear functionals on $A(G)$ introduced by Granirer \cite{G}. Using the extended module actions of $\TCQ$ on $\BLTQ$, it was shown in \cite[Proposition 5.3]{HNR1} that
\begin{align*}&\LUC=\la\LUC\star\LOQ\ra=\la\BLTQ\rhd\TCQ\ra;\\
              &\RUC=\la\LOQ\star\RUC\ra=\la\TCQ\lhd\BLTQ\ra.\end{align*}


For every locally compact quantum group $\G$, we have the left and right Arens products $\sq$ and $\Diamond$ on $\LIQ^*=\LOQ^{**}$, which are defined by
\begin{equation*}\la m\sq n,x\ra=\la m,n\sq x\ra\hs\mathrm{and}\hs\la m\Diamond n,x\ra=\la n, x\Diamond m\ra\hs(m,n\in\LIQ^*, x\in\LIQ),\end{equation*}
where $n\sq x$ and $x\Diamond m$ are elements of $\LIQ$ given by
\begin{equation*}\la n\sq x,f\ra=\la n,x\star f\ra\hs\mathrm{and}\hs\la x\Diamond m,f\ra=\la m,f\star x\ra\hs(f\in\LOQ).\end{equation*}
Then $(\LIQ^*,\sq)$ and $(\LIQ^*,\Diamond)$ are completely contractive Banach algebras.

Given $m\in\LUC^*$, we define a bounded linear map $m_L$ on $\LIQ$ by
\begin{equation*}m_L:\LIQ\ni x\mapsto m\sq x\in\LIQ,\end{equation*}
where the product $m\sq x\in\LIQ$ is given as above, noticing that $\LIQ\star\LOQ\subseteq\LUC$. This map is completely bounded, with $\norm{m_L}_{cb}\leq\norm{m}$, and a right $\LOQ$-module map, since
\begin{equation*}\la m\sq(x\star f),g\ra=\la m,x\star(f\star g)\ra=\la m\sq x,f\star g\ra=\la(m\sq x)\star f,g\ra\hs(x\in\LIQ, f,g\in\LOQ).\end{equation*}
Therefore, $m_L$ maps $\LUC$ into $\LUC$, and so the left Arens product $\sq$ on $\LIQ^*$ induces a completely contractive multiplication on $\LUC^*$, also denoted $\sq$, so that the restriction
\begin{equation*}\LIQ^*\ni m\mapsto m|_{\LUC}\in\LUC^*\end{equation*}
is a completely contractive, multiplicative quotient map from $(\LIQ^*,\sq)$ onto $(\LUC^*,\sq)$.

Let $m\in\LUC^*$. Then, as $\LUC=\la\BLTQ\rhd\TCQ\ra$, the module map $m_L$ may be extended to a right $(\TCQ,\rhd)$-module map on $\BLTQ$ via
\begin{equation*}\la m_L(x),\om\ra=\la m,x\rhd\om\ra=\la m,(\om\ten\id)V(x\ten 1)V^*\ra\hs(x\in\BLTQ, \om\in\TCQ).\end{equation*}
In this case, we also have $\norm{m_L}_{cb}\leq\norm{m}$, and if we let $\CBTCrr$ denote the algebra of completely bounded right $(\TCQ,\rhd)$-module maps on $\BLTQ$, it follows that
\begin{equation}\label{theta}\Theta^r:\LUC^*\ni m\mapsto m_L\in\CBTCrr\end{equation}
is a weak*-weak* continuous, injective, completely contractive algebra homomorphism \cite[Proposition 6.5]{HNR1}. Moreover, \cite[Theorem 7.1]{HNR1} entails that
\begin{equation}\label{inter}\Theta^r(\LUC^*)\subseteq\CBTCrr\cap\CBLTQ,\end{equation}
where $\CBLTQ$ is the algebra of completely bounded $\LIQH$-bimodule maps on $\BLTQ$ that leave $\LIQ$ invariant. Analogously, the right Arens product $\Diamond$ induces a completely contractive Banach algebra structure on $\RUC^*$, and there exists a weak*-weak* continuous, injective, completely contractive anti-homomorphism
\begin{equation}\label{thetal}\Theta^l:\RUC^*\rightarrow~\CBTCll,\end{equation}
where $\CBTCll$ is the algebra of completely bounded left $(\TCQ,\lhd)$-module maps on $\BLTQ$.

\section{Covariant Injectivity}

In this section we introduce and study versions of injectivity of $\LIQH$ that capture fundamental properties of $\G$, such as amenability, compactness, and co-commutativity. The underlying idea is to refine injectivity through a covariance condition, by which we mean the existence of a conditional expectation respecting the natural $\TCQ$-module structure of $\BLTQ$ and $\LIQH$.

\begin{defn} For a locally compact quantum group $\G$, we say that a mapping $\Phi\in\mc{CB}(\BLTQ)$ is \e{covariant} if $\Phi(x\rhd\rho)=\Phi(x)\rhd\rho$ for all $x\in\BLTQ$ and $\rho\in\TCQ$.\end{defn}

We begin by using the representation (\ref{theta}) of $\LUC^*$ to establish a one-to-one correspondence between right invariant means on $\LUC$ and covariant conditional expectations onto $\LIQH$.

\begin{thm}\label{th} Let $\G$ be a locally compact quantum group. The following statements are equivalent:
\begin{enumerate}
\item $\G$ is amenable;
\item there is a right invariant mean on $\LUC$;
\item there is a covariant conditional expectation $E:\BLTQ\rightarrow\LIQH$.\end{enumerate}\end{thm}

\begin{proof} $(1)\Rightarrow(2)$ Restriction of a right invariant mean on $\LIQ$ yields (2).\\
$(2)\Rightarrow(3)$ Let $m\in\LUC^*$ be a right invariant mean. Then $m\sq y=\la m,y\ra1$ for all $y\in\LUC$ by right invariance, which gives
\begin{equation*}\la m\sq m,y\ra=\la m,m\sq y\ra=\la m,y\ra\la m,1\ra=\la m,y\ra.\end{equation*}
Hence, $m$ is a norm one idempotent in $\LUC^*$, making $\Theta^r(m)$ a projection of norm one in $\CBTCrr$. As such, its image is equal to its fixed points, denoted $\Htm$. First observe that $\LIQH\subseteq\Htm$ as $\Theta^r(m)(\h{x})=(\id\ten m)V(\h{x}\ten1)V^*=\h{x}$. On the other hand, as $\Theta^r(m)$ is a $\TCQ$-module map, its fixed points form a $\TCQ$-submodule of $\BLTQ$. Thus, $x\rhd\om\in\Htm$ for every $x\in\Htm$ and $\om\in\TCQ$. But if $y\in\Htm\cap\LUC$, then $y=\Theta^r(m)(y)=m\sq y=\la m,y\ra1$. Hence, if $x\in\Htm$ and $\om\in\TCQ$ then $x\rhd\om=\la m,x\rhd\om\ra1$, so that for any $\tau\in\TCQ$
\begin{align*}\la\Gam(x),\om\ten\tau\ra&=\la x,\om\rhd\tau\ra=\la x\rhd\om,\tau\ra=\la m,x\rhd\om\ra\la1,\tau\ra\\
                                       &=\la\Theta^r(m)(x),\om\ra\la 1,\tau\ra=\la x\ten 1,\om\ten\tau\ra.\end{align*}
As $\om,\tau\in\TCQ$ were arbitrary, it follows that $\Gam(x)=V(x\ten1)V^*=x\ten1$, that is, $V(x\ten1)=(x\ten1)V$. Applying the slice map $(\id\ten f)$ to both sides of this equation yields $\rho(f)x=x\rho(f)$, for all $f\in\LOQ$. Therefore $x\in\rho(\LOQ)'=\LIQH$, making $E:=\Theta^r(m)$ the required projection.\\
$(3)\Rightarrow(1)$ If $E:\BLTQ\rightarrow\LIQH$ is a conditional expectation in $\CBTCrr$, then $E(\LUC)\subseteq\LUC\cap\LIQH=\C1$. Thus, by restriction we obtain a bounded linear functional $n\in\LUC^*$ satisfying $\la n,y\ra1=E(y)$ for all $y\in\LUC$. Moreover, considering the associated map $\Theta^r(n)\in\CBTCrr$, we see that
\begin{equation*}\la E(x),\om\ra=E(x)\rhd\om=E(x\rhd\om)=\la n,x\rhd\om\ra=\la\Theta^r(n)(x),\om\ra\end{equation*}
for all $x\in\BLTQ$ and $\om\in\TCQ$. This ensures that $E=\Theta^r(n)$, so in particular we have $E(\LIQ)\subseteq\LIQ\cap\LIQH=\C1$ by (\ref{inter}). Put $m:=E|_{\LIQ}$. Then $m\in\LIQ^*$ is a state satisfying
\begin{equation*}\la m,x\star f\ra=E(x\star f)=E(x)\star f=\la m,x\ra\la 1,f\ra\end{equation*}
for every $x\in\LIQ$ and $f\in\LOQ$. Hence, $m$ is a right invariant mean on $\LIQ$.\end{proof}

\begin{cor} A locally compact group $G$ is amenable if and only if there is a covariant conditional expectation $E:\BLT\rightarrow\LG$.\end{cor}


\begin{remark} In \cite{BT}, a notion of \e{topological amenability} for locally compact quantum groups $\G$ was defined by the existence of an invariant mean on $M(C_0(\G))$. The authors then asked if this notion of amenability is equivalent to the original one. The answer was recently provided, in the affirmative, by Zobeidi \cite{Z}, generalizing the partial result of Runde in the co-amenable setting \cite[Theorem 3.6]{R}. As we always have $\LUC\subseteq M(C_0(\G))$, Theorem \ref{th} provides an alternative proof (which had been found independently from \cite{Z}) for arbitrary locally compact quantum groups.\end{remark}

There is a corresponding result involving left invariant means on $\RUC$ and conditional expectations in $\CBTCll$. We state the result for completeness and for later use, but omit the details of the proof as the argument can easily be adapted from above using the left representation (\ref{thetal}).

\begin{thm}\label{th2} Let $\G$ be a locally compact quantum group. The following statements are equivalent:
\begin{enumerate}
\item $\G$ is amenable;
\item there is a left invariant mean on $\RUC$;
\item there is a conditional expectation $E:\BLTQ\rightarrow\LIQHP$ in $\CBTCll$.\end{enumerate}\end{thm}

As $\G$ is compact if and only if it admits a left invariant mean in $\LOQ$ \cite[Proposition 3.1]{BT}, and the maps $\Theta^r(f),\Theta^l(f)\in\mc{CB}(\BLTQ)$ are normal for all $f\in\LOQ$ \cite[\S4]{JNR}, the following corollary is immediate.

\begin{cor} Let $\G$ be a locally compact quantum group. The following statements are equivalent:
\begin{enumerate}
\item $\G$ is compact;
\item there is a normal covariant conditional expectation $E:\BLTQ\rightarrow\LIQH$;
\item there is a normal conditional expectation $E:\BLTQ\rightarrow\LIQHP$ in $\CBTCll$.\end{enumerate}\end{cor}

In Theorem \ref{th}, we characterized the amenability of $\G$ by means of a conditional expectation $E:\BLTQ\rightarrow\LIQH$ commuting with the right $(\TCQ,\rhd)$-module action on $\BLTQ$. As there are three other $\TCQ$-module structures on $\BLTQ$, a natural problem is to study the existence of module projections $E:\BLTQ\rightarrow\LIQH$ in each of the remaining cases. To this end, we denote by $\CBTClr$ (respectively, $\CBTCrl$) the algebra of completely bounded left $(\TCQ,\rhd)$-module (respectively, right $(\TCQ,\lhd)$-module) maps on $\BLTQ$, and for any subset $\mc{S}$ of $\mc{CB}(\BLTQ)$, we denote its commutant in $\mc{CB}(\BLTQ)$ by $\mc{S}^c$.

\begin{thm}\label{th3} Let $\G$ be a locally compact quantum group. There exists a conditional expectation $E:\BLTQ\rightarrow\LIQH$ in $\CBTClr$ if and only if $\G$ is amenable.\end{thm}

\begin{proof} By restriction, we may view any $f\in\LOQ\subseteq\LIQ^*$ as an element of $\LUC^*$. Moreover, if $\pi:(\TCQ,\rhd)\rightarrow(\LOQ,\star)$ denotes the restriction map (\ref{pi}), for $\om,\rho\in\TCQ$ and $x\in\BLTQ$ we have
\begin{equation*}\la\rho\rhd x,\om\ra=\la x,\om\rhd\rho\ra=\la x\rhd\om,\rho\ra=\la x\rhd\om,\pi(\rho)\ra=\la\Theta^r(\pi(\rho))(x),\om\ra.\end{equation*}
Thus, $\rho\rhd x=\Theta^r(\pi(\rho))(x)$, so that a map $\Phi\in\mc{CB}(\BLTQ)$ is a left $(\TCQ,\rhd)$-module homomorphism if and only if $\Phi\in\Theta^r(\LOQ)^c$.

If $\G$ is amenable, then there exists a two-sided invariant mean $m$ on $\LIQ$. Denoting again by $m$ its restriction to $\LUC$, it follows that
\begin{equation}\label{twosided}m\sq f=f\sq m=\la f,1\ra m\end{equation}
for every $f\in\LOQ$. Hence, $\Theta^r(m)\in\Theta^r(\LOQ)^c$ by (\ref{theta}). As $m$ is also a right invariant mean on $\LUC$, it follows from the proof of Theorem \ref{th} that $\Theta^r(m)$ is a conditional expectation onto $\LIQH$.

Conversely, suppose that there exists a conditional expectation $E:\BLTQ\rightarrow\LIQH$ in $\CBTClr$, and let $\h{f}\in\LOQH$ be a state. For $\om\in\TCQ$ with $f=\pi(\om)\in\LOQ$ and $x\in\LIQ$, the relations (\ref{trac}) imply
\begin{equation*}\la\h{f}\circ E, f\star x\ra=\la\h{f}\circ E,\om\rhd x\ra=\la\h{f},\om\rhd E(x)\ra=\la\om,1\ra\la\h{f}\circ E,x\ra=\la f,1\ra\la\h{f}\circ E,x\ra.\end{equation*}
Thus, $\h{f}\circ E$ is a left invariant mean on $\LIQ$.\end{proof}


\begin{remark} We note that the existence of a conditional expectation with the module property as in Theorem \ref{th3} is equivalent to the amenability of the right fundamental unitary $V$, as defined by B\'{e}dos and Tuset in \cite[\S4]{BT}.\end{remark}

\begin{thm}\label{automatic} Let $\G$ be a locally compact quantum group. There exists a conditional expectation $E:\BLTQ\rightarrow\LIQH$ in $\CBTCrl$ if and only if $\LIQH$ is injective.\end{thm}

\begin{proof} Suppose that $\h{\G}$ is injective. Then there exists a conditional expectation $E:\BLTQ\rightarrow\LIQH$. By \cite{To}, $E$ is an $\LIQH$-bimodule map on $\BLTQ$. We will show that it also lies in $\CBTCrl$. To this end, observe that a map $\Phi\in\mc{CB}(\BLTQ)$ is a right $(\TCQ,\lhd)$-module map if and only if $\Phi\in\Theta^l(\LOQ)^c$ (see the proof of Theorem \ref{th3}). If $f\in\LOQ$, then from \cite[Theorem 4.10]{JNR}, $\Theta^l(f)$ is a normal completely bounded $\LIQHP$-bimodule map on $\BLTQ$, which by an unpublished result of Haagerup \cite{Ha} implies the existence of two nets $(\h{a}_i)_{i\in I}$ and $(\h{b}_i)_{i\in I}$ in $\LIQH$ such that
\begin{equation*}\Theta^l(f)(x)=\sum_{i\in I}\h{a}_ix\h{b}_i,\end{equation*}
where the sum converges in the weak* topology of $\BLTQ$ for all $x\in\BLTQ$. Now, it follows from \cite[Lemma 2.3]{MNW} that we may approximate $E$ in the weak* topology of $\mc{CB}(\BLTQ)$ by a net of normal completely bounded $\LIQH$-bimodule maps $(\Phi_j)_{j\in J}$. Consequently, for $x\in\BLTQ$ and $\rho\in\TCQ$,
\begin{align*}\la E(\Theta^l(f)(x)),\rho\ra&=\lim_{j\in J}\la\Phi_j(\Theta^l(f)(x)),\rho\ra=\lim_{j\in J}\sum_{i\in I}\la\Phi_j(\h{a}_ix\h{b}_i),\rho\ra\\
                                           &=\lim_{j\in J}\sum_{i\in I}\la\h{a}_i\Phi_j(x)\h{b}_i,\rho\ra=\lim_{j\in J}\la\Theta^l(f)(\Phi_j(x)),\rho\ra\\
                                           &=\la\Theta^l(f)(E(x)),\rho\ra.\end{align*}
Since $f\in\LOQ$ was arbitrary, we have $E\in\Theta^l(\LOQ)^c$. As the converse is trivial, we are done.\end{proof}

\begin{cor}\label{corr} Let $\G$ be a locally compact quantum group for which there exists a state $m\in\BLTQ^*$ satisfying $m(\rho\rhd x)=m(x\lhd\rho)$ for all $x\in\BLTQ$ and $\rho\in\TCQ$. Then $\G$ is amenable if and only if $\LIQH$ is injective.\end{cor}

\begin{proof} Suppose $E:\BLTQ\rightarrow\LIQH$ is a conditional expectation. Then by Theorem \ref{automatic}, $E$ is a right $(\TCQ,\lhd)$-module map. Thus, $n:=m\circ E$ is a state on $\BLTQ$ satisfying $n(x\lhd\rho)=m(E(x)\lhd\rho)=m(\rho\rhd E(x))=\la\rho,1\ra n(x)$ for all $x\in\BLTQ$ and $\rho\in\TCQ$. It follows that $n|_{\LIQ}$ is a right invariant mean on $\LIQ$.\end{proof}

\begin{remark} The above condition, i.e., the existence of a state $m\in\BLTQ^*$ such that $m(\rho\rhd x)=m(x\lhd\rho)$ for all $x\in\BLTQ$ and $\rho\in\TCQ$, may be seen as a form of inner amenability for locally compact quantum groups. Indeed, a commutative quantum group $\G_a$ satisfies this property if and only if its underlying group $G$ is inner amenable \cite{CN}, i.e., there exists a state $n\in\LI^*$ satisfying $n(\delta_s\ast f\ast\delta_{s^{-1}})=n(f)$ for all $f\in\LI$ and $s\in G$. Moreover, one can show that discrete Kac algebras are inner amenable, thus Corollary \ref{corr} entails the equivalence of injectivity and amenability for discrete Kac algebras -- a concrete application of our techniques. The latter important result, due to Ruan \cite{Ru}, has also been derived in \cite{SV}. We hope to be able to use our approach to extend this equivalence to arbitrary discrete quantum groups, an outstanding open problem in the area.\end{remark}

\begin{thm} Let $\G$ be a locally compact quantum group. There exists a conditional expectation $E:\BLTQ\rightarrow\LIQH$ in $\CBTCll$ if and only if $\G$ is co-commutative, i.e., $\LIQ=\LG$ for some locally compact group $G$.\end{thm}

\begin{proof} If $\G$ is co-commutative, then $\LIQ=\LG$ for some locally compact group $G$, and by \cite[Theorem 4]{Re} there exists a left invariant mean $m$ on $\LG$. In this case, its restriction to $\mathrm{UCB}(\h{G})=\RUC$ is also a left invariant mean, and Theorem \ref{th2} provides a conditional expectation $\Theta^l(m):\BLTQ\rightarrow\LIQHP$ in $\CBTCll$. By duality, $\LIQH=\LI=\LIQHP$, making $\Theta^l(m)$ the desired projection.

If $E:\BLTQ\rightarrow\LIQH$ exists in $\CBTCll$, then a simple calculation implies that $(E\ten\id)\circ\Gam^l=\Gam^l\circ E$. As $\Gam^l(\cdot)=W^*(1\ten(\cdot))W$, with $W\in\LIQ\oten\LIQH$, and $E(\BLTQ)=\LIQH$, we must have $(E\ten\id)\circ\Gam^l(x)=\Gam^l\circ E(x)\in\LIQ\oten\LIQH$ for every $x\in\BLTQ$. In particular, for $\h{x}'\in\LIQHP$, we have
\begin{equation*}(E\ten\id)\circ\Gam^l(\h{x}')=(E\ten\id)(W^*(1\ten\h{x}')W)=1\ten\h{x}'\in\LIQ\oten\LIQH,\end{equation*}
implying that $\LIQHP\subseteq\LIQH$. As $\LIQH$ is in standard form on $\BLTQ$, there exists a conjugate linear isometric involution $\h{J}$ on $\LTQ$ satisfying $\h{J}\LIQH\h{J}=\LIQHP$. We therefore obtain $\LIQH\subseteq\LIQHP$, that is, $\LIQH$ is commutative. By \cite[Theorem 2; \S2]{T,VV}, $\LIQH=\LI$ for some locally compact group $G$, making $\LIQ$ co-commutative.\end{proof}

\begin{remark}\label{remar} By the proof of \cite[Theorem 2.1]{Ru}, it follows that a locally compact quantum group $\G$ is amenable if and only if there exists a \e{non-zero} left (respectively, right, two-sided) invariant functional $m\in\LIQ^*$. Hence, the existence of a \e{completely bounded} covariant projection $E:\BLTQ\rightarrow\LIQH$ is equivalent to the amenability of $\G$, and it follows that we may replace ``conditional expectation'' by ``completely bounded projection'' in the statement of every theorem and corollary in this section. For Theorem 4.10, recall that a von Neumann algebra $M\subseteq\BH$ is injective if and only if there exists a completely bounded projection $E:\BH\rightarrow M$ \cite{CS,Pi}.\end{remark}

\section{Injective Modules}

Continuing in the spirit of the previous sections, here we establish a perfect duality between quantum group amenability and injectivity in the category of $\TCQ$-modules. We also show that both amenability of $\G$ and of $\h{\G}$ may be characterized through the injectivity of $\BLTQ$ as a left, respectively, right $\TCQ$-module. This marks the starting point for subsequent work on homological properties of $\TCQ$-modules and their connections to amenability.

Let $\mc{A}$ be a completely contractive Banach algebra and $X$ be an operator space. We say that $X$ is a right \e{operator $\mc{A}$-module} if it is a right Banach $\mc{A}$-module for which the module map $m:X\hten\mc{A}\rightarrow X$ is completely contractive. We denote by $\mathbf{mod}-\mc{A}$ the category of right operator $\mc{A}$-modules with morphisms given by completely contractive module homomorphisms. If $X,Y\in\mathbf{mod}-\mc{A}$, an injective morphism $\Phi:X\rightarrow Y$ is called \e{admissible} if there exists a completely contractive map (not necessarily a morphism) $\Psi:Y\rightarrow X$ such that $\Psi\circ\Phi=\id_{X}$. An operator module $X\in\mathbf{mod}-\mc{A}$ is \e{faithful} if for every non-zero $x\in X$, there is a non-zero $a\in\mc{A}$ such that $x\cdot a\neq 0$, and $X$ is said to be \e{injective} if for every $Y,Z\in\mathbf{mod}-\mc{A}$, every injective admissible morphism $\Phi:Y\rightarrow Z$, and every morphism $\Psi:Y\rightarrow X$, there exists a morphism $\tilde{\Psi}:Z\rightarrow X$ such that $\tilde{\Psi}\circ\Phi=\Psi$. Left operator $\mc{A}$-modules are defined similarly, and there are analogous notions of admissibility, faithfulness and injectivity in this category, denoted by $\mc{A}-\mathbf{mod}$.

Let $X\in\mathbf{mod}-\mc{A}$. The unitization of $\mc{A}$, denoted $\mc{A}^+$, carries a natural operator space structure turning it into a completely contractive Banach algebra (cf. \cite[\S3.2]{W}), and it follows that $X$ becomes a right operator $\mc{A}^+$-module via the extended action
\begin{equation*}x\cdot(a+\lm e)=x\cdot a+\lm x\hs(a\in\mc{A}^+, \lm\in\C, x\in X).\end{equation*}
Then there is a canonical morphism $\Delta^+:X\rightarrow\mc{CB}(\mc{A}^+,X)$ given by
\begin{equation*}\Delta^+(x)(a)=x\cdot a\hs(x\in X, a\in\mc{A}^+),\end{equation*}
where the $\mc{A}$-bimodule structure on $\mc{CB}(\mc{A}^+,X)$ is defined by
\begin{equation*}(a\cdot\Psi)(b)=\Psi(ba)\hs\mathrm{and}\hs(\Psi\cdot a)(b)=\Psi(ab)\hs(a\in\mc{A}, \Psi\in\mc{CB}(\mc{A}^+,X), b\in\mc{A}^+).\end{equation*}
By the standard argument, it follows that $X$ is injective if and only if there exists a morphism $\Phi:\mc{CB}(\mc{A}^+,X)\rightarrow X$ that is a left inverse to $\Delta^+$. Moreover, if $X$ is faithful, by the operator space version of \cite[Proposition 1.7]{DP} (which can be proved using the operator space structure of $\mc{A}^+$, cf. \cite[Proposition 3.2.7]{W}), $X$ is injective if and only if there exists a morphism $\Phi:\mc{CB}(\mc{A},X)\rightarrow X$ that is a left inverse to $\Delta:X\rightarrow\mc{CB}(\mc{A},X)$, where $\Delta(x)(a):=\Delta^+(x)(a)$ for all $x\in X$ and $a\in\mc{A}$.

\begin{thm}\label{dualinj} Let $\G$ be a locally compact quantum group. The following statements are equivalent:
\begin{enumerate}
\item $\G$ is amenable;
\item $\LIQH$ is injective in $\mathbf{mod}-(\TCQ,\rhd)$;
\item $\LIQH$ is injective in $(\TCQ,\rhd)-\mathbf{mod}$.\end{enumerate}\end{thm}

\begin{proof} $(1)\Rightarrow(2)$ Observe that if $\h{x}\in\LIQH$ such that $0=\h{x}\rhd\rho=\la\h{x},\rho\ra1$ for all $\rho\in\TCQ$, then $\la\h{x},\h{f}\ra=0$ for all $\h{f}\in\LOQH$, making $\h{x}=0$. Thus, $\LIQH$ is faithful in $\mathbf{mod}-(\TCQ,\rhd)$. It therefore suffices to provide a morphism which is a left inverse to the map $\Delta^r:\LIQH\rightarrow\mc{CB}(\TCQ,\LIQH)$ given by
\begin{equation*}\Delta^r(\h{x})(\rho)=\h{x}\rhd\rho\hs(\h{x}\in\LIQH, \rho\in\TCQ).\end{equation*}
Identifying $\mc{CB}(\TCQ,\LIQH)\cong\BLTQ\oten\LIQH$ (cf. \cite{ER}) via
\begin{equation*}\la\Psi,\rho\ten\h{f}\ra=\la\Psi(\rho),\h{f}\ra\hs(\Psi\in\mc{CB}(\TCQ,\LIQH), \rho\in\TCQ, \h{f}\in\LOQH),\end{equation*}
one easily sees that $\Delta^r(\h{x})=\h{x}\ten 1$ for all $\h{x}\in\LIQH$, and the corresponding $\TCQ$-module structure on $\BLTQ\oten\LIQH$ is given by $T\unrhd\rho=(\rho\ten\id\ten\id)(\Gam^r\ten\id)(T)$ for $T\in\BLTQ\oten\LIQH$ and $\rho\in\TCQ$.
Since $\G$ is amenable, there exists a conditional expectation $E:\BLTQ\rightarrow\LIQH$ that is a morphism in $\mathbf{mod}-(\TCQ,\rhd)$ (cf. Theorem \ref{th}). Fix a state $\h{f}\in\LOQH$, and define $\Phi^r:\BLTQ\oten\LIQH\rightarrow\LIQH$ by
\begin{equation*}\Phi^r(T)=E((\id\ten\h{f})T)\hs(T\in\BLTQ\oten\LIQH).\end{equation*}
Then $\Phi^r$ is a complete contraction, and for $\h{x}\in\LIQH$ we have
\begin{equation*}\Phi^r(\Delta^r(\h{x}))=\Phi^r(\h{x}\ten 1)=E(\h{x})=\h{x},\end{equation*}
so that $\Phi^r$ is a left inverse to $\Delta^r$. Moreover, for $T\in\BLTQ\oten\LIQH$ and $\rho\in\TCQ$, we have
\begin{align*}\Phi^r(T\unrhd\rho)&=\Phi^r((\rho\ten\id\ten\id)(\Gam^r\ten\id)(T))=E((\rho\ten\id)\Gam^r((\id\ten\h{f})T))=E(((\id\ten\h{f})T)\rhd\rho)\\
                             &=E((\id\ten\h{f})T)\rhd\rho=\Phi^r(T)\rhd\rho.\end{align*}
$(2)\Rightarrow(1)$ If $\LIQH$ is injective in $\mathbf{mod}-(\TCQ,\rhd)$, there is a morphism $\Phi^r:\BLTQ\oten\LIQH\rightarrow\LIQH$ that is a left inverse to $\Delta^r$. Define $E:\BLTQ\rightarrow\LIQH$ by $E(x)=\Phi^r(x\ten 1)$ for all $x\in\BLTQ$. Then $E$ is a morphism, and for $\h{x}\in\LIQH$ we get
\begin{equation*}E(\h{x})=\Phi^r(\h{x}\ten1)=\Phi^r(\Delta^r(\h{x}))=\h{x},\end{equation*}
making $E$ a projection of norm one onto $\LIQH$. Theorem \ref{th} then entails the amenability of $\G$.

$(1)\Rightarrow(3)$ As above, it follows that $\LIQH$ is faithful in $(\TCQ,\rhd)-\mathbf{mod}$. We therefore have to provide a morphism which is a left inverse to $\Delta^l:\LIQH\rightarrow\mc{CB}(\TCQ,\LIQH)$ given by
\begin{equation*}\Delta^l(\h{x})(\rho)=\rho\rhd\h{x}\hs(\h{x}\in\LIQH, \rho\in\TCQ).\end{equation*}
With the identification $\mc{CB}(\TCQ,\LIQH)\cong\BLTQ\oten\LIQH$, it follows that $\Delta^l(\h{x})=1\ten\h{x}$ for all $\h{x}\in\LIQH$ and that the corresponding $\TCQ$-module structure on $\BLTQ\oten\LIQH$ is given by $\rho\unrhd T=(\id\ten\rho\ten\id)(\Gam^r\ten\id)(T)$ for $T\in\BLTQ\oten\LIQH$ and $\rho\in\TCQ$.
By amenability of $\G$, there exists a conditional expectation $E:\BLTQ\rightarrow\LIQH$ that is a morphism in $(\TCQ,\rhd)-\mathbf{mod}$ (cf. Theorem \ref{th3}). Fix a state $\h{f}\in\LOQH$, put $m:=\h{f}\circ E\in\BLTQ^*$, and define $\Phi^l:\BLTQ\oten\LIQH\rightarrow\LIQH$ by
\begin{equation*}\Phi^l(T)=(m\ten\id)(T)\hs(T\in\BLTQ\oten\LIQH).\end{equation*}
Clearly $\Phi^l$ is a completely contractive left inverse to $\Delta^l$. Furthermore, for $T\in\BLTQ\oten\LIQH$, $\rho\in\TCQ$, and $\h{g}\in\LOQH$, we have
\begin{align*}\la\Phi^l(\rho\unrhd T),\h{g}\ra&=\la(m\ten\id)(\rho\unrhd T),\h{g}\ra=\la m,(\id\ten\rho)\Gam^r((\id\ten\h{g})T)\ra\\
                                            &=\la\h{f},E(\rho\rhd((\id\ten\h{g})T))\ra=\la\h{f},\rho\rhd E((\id\ten\h{g})T)\ra\\
                                            &=\la\rho,1\ra\la m,(\id\ten\h{g})(T)\ra=\la\rho,1\ra\la\Phi^l(T),\h{g}\ra=\la\rho\rhd\Phi^l(T),\h{g}\ra.\end{align*}
$(3)\Rightarrow(1)$ If $\LIQH$ is injective in $(\TCQ,\rhd)-\mathbf{mod}$, there is a morphism $\Phi^l:\BLTQ\oten\LIQH\rightarrow\LIQH$ that is a left inverse to $\Delta^l$. Define $E:\BLTQ\rightarrow\LIQH$ by $E(x)=\Phi^l(x\ten1)$ for all $x\in\BLTQ$. Since $E(1)=\Phi^l(1\ten1)=\Phi^l(\Delta^l(1))=1$, $E$ is a unital morphism, and for any state $\h{f}\in\LOQH$, we have $1=\h{f}\circ E(1)\leq\norm{\h{f}\circ E}\leq1$, making $\h{f}\circ E$ a state in $\BLTQ^*$. By the proof of Theorem \ref{th3}, it then follows that the restriction of $\h{f}\circ E$ to $\LIQ$ is a left invariant mean.\end{proof}

\begin{remark}\label{remar2} By the observations in Remark \ref{remar}, it follows that $\G$ is amenable if and only if $\LIQH$ is injective as a right (respectively, left) module in the category of right (respectively, left) operator $(\TCQ,\rhd)$-modules with \e{completely bounded} module homomorphisms. Note, however, that injectivity in this category is formally weaker than injectivity in $\mathbf{mod}-(\TCQ,\rhd)$, (respectively, $(\TCQ,\rhd)-\mathbf{mod}$).\end{remark}

By restricting to the commutative setting, we immediately obtain a new characterization of classical amenability, while, on the other hand, restricting to the co-commutative case, we see that $\LI$ is an injective $\TC$-module for any locally compact group $G$.

\begin{cor} Let $G$ be a locally compact group. The following statements are equivalent:
\begin{enumerate}
\item $G$ is amenable;
\item $\LG$ is injective in $\mathbf{mod}-(\TC,\rhd)$;
\item $\LG$ is injective in $(\TC,\rhd)-\mathbf{mod}$.\end{enumerate}\end{cor}

\begin{cor} Let $G$ be a locally compact group. Then $\LI$ is injective in both $\mathbf{mod}-(\TC,\rhd)$ and $(\TC,\rhd)-\mathbf{mod}$.\end{cor}

Recall that the multiplication in $(\TCQ,\rhd)$ is a complete quotient map for any locally compact quantum group $\G$. Consequently, $\TCQ=\la\TCQ\rhd\TCQ\ra$, and so if $x\in\BLTQ$ satisfies $\rho\rhd x=0$ for all $\rho\in\TCQ$, then $\la\rho\rhd x,\om\ra=\la x,\om\rhd\rho\ra=0$ for all $\rho,\om\in\TCQ$, making $x=0$. Thus, $\BLTQ$ is faithful in $(\TCQ,\rhd)-\mathbf{mod}$. By a similar argument it follows that $\BLTQ$ is also faithful in $\mathbf{mod}-(\TCQ,\rhd)$.

\begin{thm}\label{amenable} Let $\G$ be a locally compact quantum group. Then $\G$ is amenable if and only if $\BLTQ$ is injective in $(\TCQ,\rhd)-\mathbf{mod}$.\end{thm}

\begin{proof} Suppose $\G$ is amenable, and let $m\in\LIQ^*$ be a two-sided invariant mean. Since $\BLTQ$ is faithful, it suffices to provide a morphism that is a left inverse for the map $\Delta:\BLTQ\rightarrow\mc{CB}(\TCQ,\BLTQ)$ given by
\begin{equation}\label{Delta}\Delta(x)(\rho)=\rho\rhd x\hs(x\in\BLTQ, \rho\in\TCQ).\end{equation}
Identifying $\mc{CB}(\TCQ,\BLTQ)\cong\BLTQ\oten\BLTQ$ (cf. \cite{ER}) via
\begin{equation*}\la\Phi,\rho\ten\om\ra=\la\Phi(\om),\rho\ra\hs(\Phi\in\mc{CB}(\TCQ,\BLTQ), \rho,\om\in\TCQ),\end{equation*}
one easily sees that $\Delta=\Gam^r$, and that the left $(\TCQ,\rhd)$-module structure on $\BLTQ\oten\BLTQ$ is given by $\rho\unrhd T=(\id\ten\id\ten\rho)(\id\ten\Gam^r)(T)$, for $T\in\BLTQ\oten\BLTQ$ and $\rho\in\TCQ$. We are therefore reduced to finding a morphism $\Phi:\BLTQ\oten\BLTQ\rightarrow\BLTQ$ satisfying $\Phi\circ\Gam^r=\id_{\BLTQ}$.

Let $n\in\TCQ$ be a state. Then $m_n:=n\circ\Theta^r(m)$ is a state on $\BLTQ$, and we define $\Phi:\BLTQ\oten\BLTQ\rightarrow\BLTQ$ by
\begin{equation}\label{normal}\Phi(T)=(\id\ten m_n)(V^*TV)\hs(T\in\BLTQ\oten\BLTQ).\end{equation}
Clearly, $\Phi$ is a complete contraction, and for $x\in\BLTQ$, we have
\begin{equation*}\Phi(\Gam^r(x))=\Phi(V(x\ten1)V^*)=(\id\ten m_n(x\ten1))=x,\end{equation*}
so $\Phi$ is a left inverse for $\Gam^r$. To show the module property, fix $T\in\BLTQ\oten\BLTQ$ and $\rho\in\TCQ$. Then, using the standard leg notation, we obtain
\begin{align*}\Phi(\rho\rhd T)&=\Phi((\id\ten\id\ten\rho)(V_{23}T_{12}V_{23}^*))=(\id\ten m_n\ten\rho)(V_{12}^*V_{23}T_{12}V_{23}^*V_{12})\\
                              &=(\id\ten m_n\ten\rho)(V_{13}V_{23}V_{12}^*T_{12}V_{12}V_{23}^*V_{13}^*)\\
                              &=(\id\ten\rho)(V(\id\ten m_n\ten\id)(V_{23}V_{12}^*T_{12}V_{12}V_{23}^*)V^*).\end{align*}
Now, for any $\tau,\om\in\TCQ$, recalling that $\pi:\TCQ\rightarrow\LOQ$ denotes the canonical quotient map, we have

\begin{align*}\la(\id\ten m_n\ten\id)(V_{23}V_{12}^*T_{12}V_{12}V_{23}^*),\tau\ten\om\ra&=\la(m_n\ten\id)V((\tau\ten\id)(V^*TV)\ten1)V^*,\om\ra\\
                                                                          &=\la m_n,\Theta^r(\pi(\om))((\tau\ten\id)V^*TV)\ra\\
                                                                          &=\la n,\la\om,1\ra\Theta^r(m)((\tau\ten\id)V^*TV)\ra\hs(\mathrm{by\hskip2pt equation\hskip2pt (\ref{twosided})})\\
                                                                          &=\la m_n\ten\om,(\tau\ten\id)(V^*TV)\ten 1\ra\\
                                                                          &=\la(\id\ten m_n\ten\id)(V^*TV\ten 1),\tau\ten\om\ra.\\\end{align*}
Since $\tau$ and $\om$ in $\TCQ$ were arbitrary, it follows that
\begin{align*}\Phi(\rho\rhd T)&=(\id\ten\rho)(V(\id\ten m_n\ten\id)(V_{23}V_{12}^*T_{12}V_{12}V_{23}^*)V^*)\\
                              &=(\id\ten\rho)(V(\id\ten m_n\ten\id)(V^*TV\ten 1)V^*)\\
                              &=(\id\ten\rho)(V(\Phi(T)\ten 1)V^*)=\rho\rhd\Phi(T).\end{align*}

Conversely, suppose that there exists a morphism $\Phi:\BLTQ\oten\BLTQ\rightarrow\BLTQ$ that is a left inverse to $\Gam^r$. Then $\Gam^r\circ\Phi$ is a conditional expectation onto the image of $\Gam^r$, and $\Gam^r\circ\Phi=(\Phi\ten\id)(\id\ten\Gam^r)$ as $\Phi$ is a module map. Define a map $E:\BLTQ\rightarrow\BLTQ$ by
\begin{equation*}E(x)=\Phi(x\ten 1)\hs(x\in\BLTQ).\end{equation*}
Then $E$ is a complete contraction, and for $x\in\BLTQ$ we have
\begin{equation*}\Gam^r(E(x))=\Gam^r(\Phi(x\ten1))=(\Phi\ten\id)(\id\ten\Gam^r)(x\ten 1)=(\Phi\ten\id)(x\ten 1\ten 1)=\Phi(x\ten 1)\ten 1=E(x)\ten 1,\end{equation*}
which by the standard argument shows that $E(x)\in\LIQH$. Moreover, $E(\h{x})=\Phi(\h{x}\ten 1)=\Phi(\Gam^r(\h{x}))=\h{x}$ for all $\h{x}\in\LIQH$ making $E$ a projection of norm one onto $\LIQH$.

Since $\Gam^r\circ\Phi$ is a conditional expectation onto $\Gam^r(\BLTQ)$, it follows from \cite{To} that
\begin{equation*}(\Gam^r\circ\Phi)(\Gam^r(x)T\Gam^r(y))=\Gam^r(x)(\Gam^r\circ\Phi(T))\Gam^r(y)=\Gam^r(x\Phi(T)y),\end{equation*}
which, by the injectivity of $\Gam^r$, implies $\Phi(\Gam^r(x)T\Gam^r(y))=x\Phi(T)y$, for all $x,y\in\BLTQ$ and $T\in\BLTQ\oten\BLTQ$. Taking $T=x'\ten 1\in\LIQ'\oten\LIQ'$ and $x\in\LIQ$, we therefore have $\Phi(x'\ten 1)x=x\Phi(x'\ten 1)$. Consequently, $E(x')=\Phi(x'\ten 1)\in\LIQ'$ for every $x'\in\LIQ'$. Since $\LIQ$ is standard in $\BLTQ$, there is a conjugate linear involution $J$ on $\LTQ$ satisfying $J\LIQ J=\LIQ'$. Moreover, $J\LIQH J\subseteq\LIQH$ \cite[Proposition 2.1]{KV2}, so that $E_J:\BLTQ\rightarrow\LIQH$ given by
\begin{equation*}E_J(x)=JE(JxJ)J\hs(x\in\BLTQ)\end{equation*}
also defines a conditional expectation onto $\LIQH$. Clearly, $E_J(\LIQ)\subseteq\LIQ\cap\LIQH=\C1$, so \cite[Theorem 3]{SV} entails the amenability of $\G$.\end{proof}

\begin{remark} Contrary to Theorem \ref{dualinj}, it is not immediately obvious if one can weaken the hypothesis of Theorem \ref{amenable} to injectivity in the category of left $(\TCQ,\rhd)$-modules with \e{completely bounded} morphisms (cf. Remark \ref{remar2}).\end{remark}

By considering the category of left operator $\TCQ$-modules with \e{normal} completely contractive morphisms, denoted $(\TCQ,\rhd)-\mathbf{nmod}$, we obtain the following characterization of compactness.

\begin{cor} Let $\G$ be a locally compact quantum group. Then $\G$ is compact if and only if $\BLTQ$ is injective in $(\TCQ,\rhd)-\mathbf{nmod}$.\end{cor}

\begin{proof} If $\G$ is compact then there is a two-sided invariant mean $m\in\LOQ$, and one may define a normal morphism as in equation (\ref{normal}) to produce a left inverse to $\Delta$, as defined in (\ref{Delta}). Conversely, one may repeat the second half of the proof of Theorem \ref{amenable} to obtain a normal conditional expectation from $\BLTQ$ onto $\LIQH$ mapping $\LIQ$ into $\C1$. Then \cite[Theorem 4.2]{KN} implies that $\h{\G}$ is discrete whence $\G$ is compact.\end{proof}

\begin{prop}\label{dualamenable} Let $\G$ be a locally compact quantum group. If $\h{\G}$ is amenable, then $\BLTQ$ is injective in $\mathbf{mod}-(\TCQ,\rhd)$.\end{prop}

\begin{proof} By \cite[Proposition 2.15]{KV2}, the unitary operator $U\ten U:=\h{J}J\ten\h{J}J$ on $\LTQ\ten\LTQ$ intertwines the right fundamental unitaries of $\h{\G}$ and $\h{\G}'$, denoted $\h{V}$ and $\h{V}'$, respectively. One then obtains a one-to-one correspondence between invariant means on $\LIQH$ and $\LIQHP$ via conjugation with $U$, making $\h{\G}$ amenable if and only if $\h{\G}'$ is. Thus, assuming amenability of $\h{\G}$, we let $\h{m}'$ be a two-sided invariant mean on $\LIQHP$. Similar to the previous theorem, we must provide a morphism which is a left inverse to the map $\Delta:\BLTQ\rightarrow\mc{CB}(\TCQ,\BLTQ)$ given by
\begin{equation}\label{delta}\Delta(x)(\rho)=x\rhd\rho\hs(x\in\BLTQ, \rho\in\TCQ).\end{equation}
In this case, we identify $\mc{CB}(\TCQ,\BLTQ)\cong\BLTQ\oten\BLTQ$ (cf. \cite{ER}) via
\begin{equation*}\la\Phi,\rho\ten\om\ra=\la\Phi(\rho),\om\ra\hs(\Phi\in\mc{CB}(\TCQ,\BLTQ), \rho,\om\in\TCQ).\end{equation*}
This ensures $\Delta=\Gam^r$, and that the corresponding $\TCQ$-module structure on $\BLTQ\oten\BLTQ$ is defined by $T\unrhd\rho=(\rho\ten\id\ten\id)(\Gam^r\ten\id)(T)$ for $T\in\BLTQ\oten\BLTQ$ and $\rho\in\TCQ$.

We take a normal state $n\in\TCQ$, and define $\Phi:\BLTQ\oten\BLTQ\rightarrow\BLTQ$ by
\begin{equation*}\Phi(T)=(\id\ten m_n)(V^*TV)\hs(T\in\BLTQ\oten\BLTQ),\end{equation*}
where $m_n:=n\circ\Theta^r(\h{m}')$ is a state on $\BLTQ$, and $\Theta^r$ denotes the representation of $\mathrm{LUC}(\h{\G}')$. Clearly, $\Phi$ is a completely contractive left inverse to $\Gam^r$. To show that $\Phi$ is also a module map we follow along similar lines as in Theorem \ref{amenable}. Fix $T\in\BLTQ\oten\BLTQ$ and $\rho\in\TCQ$. Then
\begin{align*}\Phi(T\unrhd\rho)&=\Phi((\rho\ten\id\ten\id)(V_{12}T_{13}V_{12}^*))=(\rho\ten\id\ten m_n)(V_{23}^*V_{12}T_{13}V_{12}^*V_{23})\\
                              &=(\rho\ten\id\ten m_n)(V_{12}V_{23}^*V_{13}^*T_{13}V_{13}V_{23}V_{12}^*)=(\rho\ten\id)(V(\id\ten\id\ten m_n)(V_{23}^*V_{13}^*T_{13}V_{13}V_{23})V^*).\end{align*}
Now, denoting $\pi$ by the canonical quotient map $\TCQ\rightarrow\LOQHP$, and using the fact that $\h{V}'=\sigma V^*\sigma$, where $\sigma$ is the flip map on $\LTQ\ten\LTQ$, for any $\tau,\om\in\TCQ$, we have
\begin{align*}\la(\id\ten\id\ten\h{m'}_n)(V_{23}^*V_{13}^*T_{13}V_{13}V_{23}),\tau\ten\om\ra&=\la(\id\ten\id\ten m_n)(V_{23}^*(\sigma\ten 1)V_{23}^*T_{23}V_{23}(\sigma\ten1)V_{23}),\tau\ten\om\ra\\
&=\la(\id\ten\id\ten m_n)(V_{13}^*V_{23}^*T_{23}V_{23}V_{13}),\om\ten\tau\ra\\
&=\la(\id\ten m_n)(V^*(1\ten(\tau\ten\id)(V^*TV))V),\om\ra\\
&=\la( m_n\ten\id)(\h{V}'((\tau\ten\id)(V^*TV)\ten1)\h{V}'^*),\om\ra\\
&=\la m_n,\Theta^r(\pi(\om))((\tau\ten\id)(V^*TV))\ra\\
&=\la n,\la\om,1\ra\Theta^r(\h{m}')((\tau\ten\id)(V^*TV))\ra\\
&=\la m_n\ten\om,(\tau\ten\id)(V^*TV)\ten 1\ra\\
&=\la(\id\ten m_n\ten\id)(V^*TV\ten 1),\tau\ten\om\ra.\end{align*}
As $\tau$ and $\om$ were arbitrary, we have
\begin{align*}\Phi(T\unrhd\rho)&=(\rho\ten\id)(V(\id\ten\id\ten m_n)(V_{23}^*V_{13}^*T_{13}V_{13}V_{23})V^*)\\
                             &=(\rho\ten\id)(V(\id\ten m_n\ten\id)(V^*TV\ten 1)V^*)\\
                             &=(\rho\ten\id)(V(\Phi(T)\ten1)V^*)\\
                             &=\Phi(T)\rhd\rho.\end{align*}\end{proof}

\begin{remark} We remark that the converse of Proposition~\ref{dualamenable} holds in the setting of Kac algebras. The proof involves machinery from representations of completely bounded multipliers over quantum groups \cite{JNR}, and will therefore appear in subsequent work.\end{remark}

\begin{remark} Let $\mc{A}$ be a completely contractive Banach algebra. We say that an element $X\in\mc{A}-\mathbf{mod}$ is \e{flat} if and only if $X^*$ is injective in $\mathbf{mod}-\mc{A}$. In this case, we obtain a stronger version of operator flatness than the usual definition involving completely bounded morphisms \cite{A} (see \cite{H} for the classical setting). If $\G$ is a locally compact quantum group such that $\h{\G}$ is amenable, then Proposition \ref{dualamenable} implies that $(\TCQ,\rhd)$ is flat in $(\TCQ,\rhd)-\mathbf{mod}$, a property solely in terms of the Banach algebra structure of $(\TCQ,\rhd)$, without involving the dual quantum group. If one can deduce the existence of a bounded right approximate identity in $(\TCQ,\rhd)$, which is true in the case of $\G_s$, then $\G$ is co-amenable by \cite[Proposition 5.4]{HNR1}. Since, on the other hand, co-amenability of $\G$ implies amenability of $\h{\G}$ by \cite[Theorem 3.2]{BT}, one would thus obtain a solution to the long-standing conjecture on the duality of amenability and co-amenability for arbitrary locally compact quantum groups. Moreover, the results in this paper suggest further approaches to this open problem.\end{remark}

\section*{Acknowledgements}

This work was completed as part of the doctoral thesis of the first author, who was supported by an NSERC Canada Graduate Scholarship and a FCRF Joint PhD Scholarship. The second author was partially supported by an NSERC Discovery Grant.

\end{spacing}

\end{document}